\documentclass{article}

\usepackage{amsmath,amssymb,amsthm}
\usepackage{hyperref}
\usepackage{mathrsfs}
\usepackage{footmisc}

\theoremstyle{plain}
\newtheorem{theorem}{Theorem}
\newtheorem{lemma}{Lemma}
\newtheorem{problem}{Problem}
\newtheorem{proposition}{Proposition}
\theoremstyle{definition}
\newtheorem{remark}{Remark}
\usepackage{graphicx}

\renewcommand{\Re}{\mathop{\mathrm{Re}}}
\renewcommand{\Im}{\mathop{\mathrm{Im}}}
\DeclareMathOperator{\supp}{supp}

\renewcommand{\thefootnote}{\fnsymbol{footnote}}

\begin{document}
\begin{center}\LARGE Error estimates for phaseless inverse scattering in the Born approximation at high energies\footnote[1]{Dedicated to G. M. Henkin}\end{center}
\renewcommand*{\thefootnote}{\arabic{footnote}}
\begin{center}\it A. D. Agaltsov\footnote{CMAP, Ecole Polytechnique, CNRS, Universit\'e Paris-Saclay, 91128, Palaiseau, France; email: agaltsov@cmap.polytechnique.fr} and R. G. Novikov\footnote{CMAP, Ecole Polytechnique, CNRS, Universit\'e Paris-Saclay, 91128, Palaiseau, France; IEPT RAS, 117997 Moscow, Russia; email: novikov@cmap.polytechnique.fr}\end{center}
\begin{center} \today \end{center}

\begin{quote} \textbf{Abstract.} We study explicit formulas for phaseless inverse scattering in the Born approximation at high energies for the Schr\"odinger equation with compactly supported potential in dimension $d \geq 2$. We obtain error estimates for these formulas in the configuration space.
\end{quote}

\section{Introduction}
We consider the time-independent Schr\"odinger equation
\begin{equation}\label{in.schreq}
    -\Delta \psi + v(x)\psi = E\psi, \quad x \in \mathbb R^d, \; d \geq 2, \; E > 0,
\end{equation}
where
\begin{equation}
    v \in L^\infty(\mathbb R^d), \quad \supp v \subset D, \label{in.vprop}
\end{equation}
where $D$ is some fixed open bounded domain in $\mathbb R^d$. 

In quantum mechanics equation \eqref{in.schreq} describes an elementary particle interacting with a macroscopic object contained in $D$ at fixed energy $E$. In this setting one usually assumes that $v$ is real-valued.

Equation \eqref{in.schreq} at fixed $E$ can be also interpreted as the Helmholtz equation of acoustics or electrodynamics. In these frameworks the coefficient $v$ can be complex-valued. In addition, the imaginary part of $v$ is related to the absorption coefficient.

For equation \eqref{in.schreq} we consider the classical scattering solutions $\psi^+ = \psi^+(x,k)$, where $x = (x_1,\dots,x_d) \in \mathbb R^d$, $k = (k_1,\dots,k_d) \in \mathbb R^d$, $k^2 = k_1^2+\cdots+k_d^2 = E$. These solutions $\psi^+$ can be specified by the following asymptotics as $|x| \to \infty$:
\begin{equation}\label{in.psi+as}
    \begin{gathered}
    \psi^+(x,k) = e^{ikx} + c(d,|k|) \frac{e^{i|k||x|}}{|x|^{(d-1)/2}} f(k,|k|\tfrac{x}{|x|}) + O(|x|^{-(d+1)/2}), \\
    x \in \mathbb R^d, \; k \in \mathbb R^d, \; k^2 = E, \; kx = k_1 x_1 + \cdots + k_d x_d, \\
    c(d,|k|) = -\pi i (-2\pi i)^{(d-1)/2} |k|^{(d-3)/2},
    \end{gathered}
\end{equation}
for some a priori unknown $f$. The function $f$ arising in \eqref{in.psi+as} is defined on 
\begin{equation}
         \mathcal M_E = \bigl\{ (k,l) \in \mathbb R^d \times \mathbb R^d \colon k^2 = l^2 = E \bigr\},
\end{equation}
and is known as the classical scattering amplitude for equation \eqref{in.schreq}.

In quantum mechanics $|f(k,l)|^2$ describes the probability density of scattering of particle with initial momentum $k$ into direction $l/|l| \neq k/|k|$, and is known as differential scattering cross section for equation \eqref{in.schreq}; see, e.g., \cite[Chapter 1, Section 6]{Fad1993}.

The problem of finding $\psi^+$ and $f$ from $v$ is known as the direct scattering problem for equation \eqref{in.schreq}. For solving this problem, one can use, in particular, the Lippmann-Schwinger integral equation for $\psi^+$ and an explicit integral formula for $f$, see, e.g., \cite{Berez1991,Fad1976,Nov2016}.

In turn, the problem of finding $v$ from $f$ is known as the inverse scattering problem (with phase information) and the problem of finding $v$ from $|f|^2$ is known as the phaseless inverse scattering problem for equation \eqref{in.schreq}.

There are many important results on the former inverse scattering problem with phase information; see \cite{Alex2008,Barc2016,Bur2009en,Chad1989,Eskin2011,Fad1956,Gonchar1992,Grin2000,Haeh2001,Henk1987,Henk1988,Isay2013c,Isay2013b,Nov1988en,Nov1998,Nov2005b,Nov2013b,Nov2015en} and references therein. In particular, it is well known that the scattering amplitude $f$ uniquely determines $v$ via the Born approximation formulas at high energies:
\begin{gather}
    \widehat v(k-l) = f(k,l) + O(E^{-\frac 1 2}), \quad E \to + \infty, \; (k,l) \in \mathcal M_E, \label{in.vborn}\\ 
    \widehat v(p) = (2\pi)^{-d} \int_{\mathbb R^d} e^{ipx} v(x) \, dx, \quad p \in \mathbb R^d,\label{in.fourier}
\end{gather}
and the inverse Fourier transform; see, e.g., \cite{Fad1956,Nov2015en}.

On the other hand, the literature for the phaseless case is much more limited; see \cite{Chad1989,Nov2016} and references therein for the case of the aforementioned phaseless inverse problem and see \cite{Klib2014,Klib2016a,Klib2016b,Nov2015LS,Nov2015b,Nov2016} and references therein for the case of some similar inverse problems without phase information. In addition, it is well known that the phaseless scattering data $|f|^2$ does not determine $v$ uniquely, even if $|f|^2$ is given completely for all positive energies. In particular, it is known that
\begin{equation}
    \begin{gathered}
    f_y(k,l) = e^{i(k-l)y} f(k,l), \\
    |f_y(k,l)|^2 = |f(k,l)|^2, \quad k,l \in \mathbb R^d, \; k^2 = l^2 > 0,
    \end{gathered}
\end{equation}
where $f$ is the scattering amplitude for $v$ and $f_y$ is the scattering amplitude for $v_y = v(\cdot-y)$, where $y \in \mathbb R^d$; see \cite{Nov2016} and references therein.

In the present work, in view of the aforementioned non-uniqueness for the problem of finding $v$ from $|f|^2$, we consider the modified phaseless inverse scattering problem formulated below as Problem \ref{in.probpisc}. Let
\begin{equation}
        S = \{ |f|^2, |f_1|^2, \ldots, |f_m|^2\},
\end{equation}
where $f$ is the scattering amplitude for $v$ and $f_1$, \ldots, $f_m$ are the scattering amplitudes for $v_1$, \dots, $v_m$, where
\begin{equation}\label{in.defvj}
        v_j = v + w_j, \quad j = 1,\ldots, m,
\end{equation}
where $w_1$, \dots, $w_m$ are additional a priori known background scatterers such that
\begin{equation}\label{in.wjass}
\begin{gathered}
        w_j \in L^\infty(\mathbb R^d), \quad \supp w_j \subset \Omega_j,\\
        \text{$\Omega_j$ is an open bounded domain in $\mathbb R^d$}, \quad \Omega_j \cap D = \varnothing,\\
        w_j \neq 0, \quad \text{$w_{j_1} \neq w_{j_2}$ if $j_1 \neq j_2$ (in $L^\infty(\mathbb R^d$)),}   
\end{gathered}
\end{equation}
where $j$, $j_1$, $j_2 \in \{1,\ldots,m\}$. Thus, $S$ consists of the phaseless scattering data $|f|^2$, $|f_1|^2$, \dots, $|f_m|^2$ measured sequentially, first, for the unknown scatterer $v$ and then for $v$ in the presence of known scatterer $w_j$ disjoint from $v$ for $j = 1$, \dots, $m$.

Actually, in the present work we continue studies of \cite{Nov2016} on the following inverse scattering problem for equation \eqref{in.probpisc}:

\begin{problem}\label{in.probpisc} Reconstruct potential $v$ from the phaseless scattering data $S$ for some appropriate background scatterers $w_1$, \ldots, $w_m$.
\end{problem}

Studies of Problem \ref{in.probpisc} in dimension $d \geq 2$ were started in \cite{Nov2016}. In dimension $d = 1$ for $m = 1$ studies of Problem \ref{in.probpisc} were started earlier in \cite{Akto1998}, where phaseless scattering data was considered for all $E>0$.

Actually, the key result of \cite{Nov2016} consists in a proper extension of formula \eqref{in.vborn} for the Fourier transform $\widehat v$ of $v$ to the phaseless case of Problem \ref{in.probpisc}; see Section~\ref{sec.ex}.

In the present work we proceed from the aforementioned result of \cite{Nov2016} and study related approximate reconstruction of $v$ in the configuration space. In this connection our results consist in obtaining related error estimates in the configuration space at high energies $E$; see Section \ref{sec.er}.

In addition, results of the present work are necessary for extending the iterative algorithm of \cite{Nov2015en} to the phaseless case of Problem \ref{in.probpisc}. The latter extension will be given in \cite{Agal2016a}.

\section{Extension of formula \eqref{in.vborn} to the phaseless case}\label{sec.ex}

Actually, the key result of \cite{Nov2016} consists in the following formulas for solving Problem \ref{in.probpisc} in dimension $d \geq 2$ for $m = 2$ at high energies $E$:

\begin{gather}
\label{in.phlBorn}
    \begin{pmatrix}
       \Re \widehat v \\
       \Im \widehat v
    \end{pmatrix} = 
    \frac 1 2 \begin{pmatrix} 
                \Re \widehat w_1 & \Im \widehat w_1 \\
                \Re \widehat w_2 & \Im \widehat w_2
              \end{pmatrix}^{-1}
    \begin{pmatrix}
        |\widehat v_1|^2 - |\widehat v|^2 - |\widehat w_1|^2 \\
        |\widehat v_2|^2 - |\widehat v|^2 - |\widehat w_2|^2
    \end{pmatrix}, \\
    \begin{gathered}\label{in.phlBornAbs}
        |\widehat v_j(p)|^2 = |f_j(k,l)|^2  + O(E^{-\frac 1 2}), \quad E \to + \infty, \\
        p \in \mathbb R^d, \quad (k,l) \in \mathcal M_E, \quad k-l = p, \quad j = 0,1,2,
    \end{gathered}
\end{gather}
where:
\begin{itemize}
 \item $v_0 = v$, $v_j$ is defined by \eqref{in.defvj}, $j = 1$, $2$, and $f_0 = f$, $f_1$, $f_2$ are the scattering amplitudes for $v_0$, $v_1$, $v_2$, respectively;
 \item $\widehat v = \widehat v(p)$, $\widehat v_j = \widehat v_j(p)$, $\widehat w_j = \widehat w_j(p)$, $p \in \mathbb R^d$, are the Fourier transforms of $v$, $v_j$, $w_j$ (defined as in \eqref{in.fourier}); 
 \item formula \eqref{in.phlBorn} is considered for all $p \in \mathbb R^d$ such that the determinant
\begin{equation}\label{in.zetanot0}
   \zeta_{\widehat w_1,\widehat w_2}(p) \overset{def}{=\joinrel=} \Re \widehat w_1(p) \Im \widehat w_2(p) - \Im \widehat w_1(p) \Re \widehat w_2(p) \neq 0.
\end{equation}
\end{itemize}

The point is that using formulas \eqref{in.phlBornAbs} for $d \geq 2$ with
\begin{equation}\label{in.kElEdef}
    \begin{gathered}
        k = k_E(p) = \tfrac p 2 + \bigl(E - \tfrac{p^2}{4}\bigr)^{1/2} \gamma(p), \\
        l = l_E(p) = -\tfrac p 2 + \bigl(E - \tfrac{p^2}{4}\bigr)^{1/2} \gamma(p), \\
        |\gamma(p)| = 1, \quad \gamma(p) p = 0,
    \end{gathered}
\end{equation}
where $p \in \mathbb R^d$, $|p| \leq 2\sqrt E$, one can reconstruct $|\widehat v|^2$, $|\widehat v_1|^2$, $|\widehat v_2|^2$ from $S$ at high energies for any $p \in \mathbb R^d$. And then using formula \eqref{in.phlBorn} one can reconstruct $\widehat v$ completely, provided that condition $\eqref{in.zetanot0}$ is fulfilled for almost all $p \in \mathbb R^d$.

\begin{remark} Formulas \eqref{in.phlBornAbs} can be precised as formula (2.15) of \cite{Nov2016}:
\begin{equation}\label{in.estconsts}
    \begin{gathered}
    \bigl| |\widehat v_j(p)|^2 - |f_j(k,l)|^2 \bigr| \leq c(D_j)N_j^3 E^{-\frac 1 2},\\
    p = k - l,\; (k,l) \in \mathcal M_E, \; E^{\frac 1 2} \geq \rho(D_j,N_j), \; j = 0,1,2,
    \end{gathered}
\end{equation}
where $\|v_j\|_{L^\infty(D_j)} \leq N_j$, $j = 0$, $1$, $2$, and $D_0 = D$, $D_j = D \cup \Omega_j$, $j = 1$, $2$, and constants $c$, $\rho$ are given by formulas (3.10) and (3.11) in \cite{Nov2016} (and, in particular, $\rho \geq 1$).
\end{remark}

In addition, from the experimental point of view it seems to be, in particular, convenient to consider Problem \ref{in.probpisc} with $m = 2$ for the case when $w_2$ is just a translation of $w_1$:
\begin{equation}
    w_2(x) = w_1(x-y), \quad x \in \mathbb R^d, \; y \in \mathbb R^d.
\end{equation}
In this case
\begin{equation}
    \widehat w_2(p) = e^{ipy} \widehat w_1(p), \quad \zeta_{\widehat w_1,\widehat w_2}(p) = \sin(py)|\widehat w_1(p)|^2, \quad p \in \mathbb R^d.
\end{equation}

On the level of analysis, the principal complication of \eqref{in.phlBorn}, \eqref{in.phlBornAbs} in comparison with \eqref{in.vborn} consists in possible zeros of the determinant $\zeta_{\widehat w_1,\widehat w_2}$ of \eqref{in.zetanot0}. For some simplest cases, we study these zeros in the next section.

\section{Zeros of the determinant $\zeta_{\widehat w_1,\widehat w_2}$}\label{sec.zeros}

Let 
\begin{equation}\label{det.defZ}
\begin{aligned}
    Z_{\widehat w_1,\widehat w_2} & = \bigl\{ p \in \mathbb R^d \colon \zeta_{\widehat w_1,\widehat w_2}(p) = 0 \bigr\}, \\
    Z_{\widehat w_j} & = \bigl\{ p \in \mathbb R^d \colon \widehat w_j(p) = 0 \bigr\}, \quad j = 1,2,
\end{aligned}
\end{equation}
where $\zeta$ is defined by \eqref{in.zetanot0}. From \eqref{in.zetanot0}, \eqref{det.defZ} it follows that 
\begin{equation}\label{det.Zint}
    Z_{\widehat w_1} \cup Z_{\widehat w_2} \subseteq Z_{\widehat w_1,\widehat w_2}.
\end{equation}
In view of \eqref{det.Zint}, in order to construct examples of $w_1$, $w_2$ such that the set $Z_{\widehat w_1,\widehat w_2}$ is as simple as possible, we use the following lemma:
\begin{lemma}\label{det.lemw} Let
\begin{gather}
    w(x) = |x|^\nu K_\nu(|x|) \int_{\mathbb R^d} q(x-y) q(y) \, dy, \quad x \in \mathbb R^d, \; \nu > 0, \label{det.defw} \\
    K_\nu(s) = \frac{\Gamma(\tfrac 1 2+\nu)}{\sqrt \pi}\left(\frac 2 s\right)^\nu \int_0^\infty \frac{\cos(s t) \, dt}{(1+t^2)^{\frac 1 2+\nu}}, \quad s > 0, \label{det.defKnu}\\
    \begin{gathered}\label{det.qprop}
        q \in L^\infty(\mathbb R^d), \; q = \overline q, \; \text{$q \neq 0$ in $L^\infty(\mathbb R^d)$},\\
        \text{$q(x) = 0$ if $|x|>r$}, \; q(x) = q(-x), \; x \in \mathbb R^d. 
    \end{gathered}
\end{gather}
Then
\begin{equation}\label{det.wprop}
    \begin{gathered}
        w \in C(\mathbb R^d), \; w = \overline w, \; \text{$w(x) = 0$ if $|x| > 2r$}, \; x \in \mathbb R^d, \\
        \widehat w(p) = \overline{\widehat w(p)}  \geq c_1 (1+|p|)^{-\beta}, \quad p \in \mathbb R^d, 
    \end{gathered}
\end{equation}
for $\beta=d+2\nu$ and some positive constant $c_1 = c_1(q,\nu)$, where $\widehat w$ is the Fourier transform of $w$. In addition, if $q \geq 0$, then $w \geq 0$.
\end{lemma}

We recall that $K_\nu$ defined by \eqref{det.defKnu} is the modified Bessel function of the second kind and order $\nu$. In addition, $\Gamma$ denotes the gamma function.

Lemma \ref{det.lemw} is proved in Section \ref{sec.l1p}.

As a corollary of Lemma \ref{det.lemw}, functions 
\begin{equation}\label{det.wjTj}
    w_j(x) = w(x-T_j), \quad x \in \mathbb R^d, \; T_j \in \mathbb R^d,
\end{equation}
where $w$ is constructed in Lemma \ref{det.lemw}, give us examples of $w_j$ satisfying \eqref{in.wjass} for fixed $D$, $\Omega_j$ and for appropriate radius $r$ of Lemma \ref{det.lemw} and translations $T_j$ of \eqref{det.wjTj}, and such that 
\begin{equation}
    \begin{gathered}
    Z_{\widehat w_j} = \varnothing, \\
    |\widehat w_j(p)| = \widehat w(p) \geq c_1(1+|p|)^{-\beta}, \quad p \in \mathbb R^d,
    \end{gathered}
\end{equation}
where $c_1$, $\beta$ are the same as in \eqref{det.wprop}. In addition,
\begin{equation}\label{det.defZw1w2}
     \begin{gathered}
         \zeta_{\widehat w_1,\widehat w_2}(p) = \sin(py) |\widehat w(p)|^2, \quad y = T_2 - T_1 \neq 0, \quad p \in \mathbb R^d, \\
         Z_{\widehat w_1,\widehat w_2} = \bigl\{ p \in \mathbb R^d \colon \sin(py) = 0 \bigr\} = \bigl\{ p \in \mathbb R^d \colon py \in \pi \mathbb Z \bigr\},
     \end{gathered}
\end{equation}
for $w_1$, $w_2$ of \eqref{det.wjTj}.

As another corollary of Lemma \ref{det.lemw}, we have that
\begin{equation}
    \begin{gathered}\label{det.w2=iw1}
      \text{if $w_1$ is defined as in \eqref{det.wjTj} and $w_2 = iw_1$, then}\\
        \zeta_{\widehat w_1,\widehat w_2}(p) = |\widehat w(p)|^2 \geq c_1^2 (1+|p|)^{-2\beta}, \quad p \in \mathbb R^d, \\
        Z_{\widehat w_1,\widehat w_2} = \varnothing.
    \end{gathered}
\end{equation}

We recall that complex-valued $v$ and $w_j$ naturally arise if we interpret equation \eqref{in.schreq} for fixed $E$ as the Helmholtz equation of acoustics or electrodynamics.

Finally, note that 
\begin{equation}\label{det.defZw1..wd}
    \begin{gathered}
    Z_{\widehat w_1,\ldots,\widehat w_{d+1}} = \tfrac{\pi}{s} \mathbb Z^d, \quad \text{where} \\
    Z_{\widehat w_1,\ldots,\widehat w_{d+1}} = Z_{\widehat w_1,\widehat w_2} \cap Z_{\widehat w_1,\widehat w_3} \cap \cdots \cap Z_{\widehat w_1,\widehat w_{d+1}},
    \end{gathered}
\end{equation}
if $w_1$ is defined as in \eqref{det.wjTj}, and
\begin{equation}\label{det.defw1wd}
w_2(x) = w_1(x-s e_1), \ldots, w_{d+1}(x) = w_1(x - s e_d),
\end{equation}
where $(e_1,\ldots,e_d)$ is the standard basis of $\mathbb R^d$ and $s>0$.

Thus, in principle, for Problem \ref{in.probpisc} with background scatterers $w_1$, \dots, $w_{d+1}$ as in \eqref{det.defw1wd}, for each $p \in \mathbb R^d \setminus \tfrac{\pi}{s}\mathbb Z^d$ formulas \eqref{in.phlBorn}, \eqref{in.phlBornAbs} can be used with appropriate $w_j$ in place of $w_2$, where $j = 2$, \dots, $d+1$.

\section{Error estimates in the configuration space}\label{sec.er}

We recall that for inverse scattering with phase information the scattering amplitude $f$ on $\mathcal M_E$ processed by \eqref{in.vborn} and the inverse Fourier transform yield the approximate reconstruction
\begin{equation}\label{er.utov}
    u(\cdot,E) = v + O(E^{-\alpha}) \quad \text{in $L^\infty(D)$ as $E\to+\infty$}, \;\; \alpha = \frac{n-d}{2n},
\end{equation}
if $v \in W^{n,1}(\mathbb R^d)$, $n > d$ (in addition to the initial assumption \eqref{in.vprop}), where $W^{n,1}(\mathbb R^d)$ denotes the standard Sobolev space of $n$-times differentiable functions in $L^1(\mathbb R^d)$:
\begin{equation}\label{er.defWn1}
  \begin{gathered}
    W^{n,1}(\mathbb R^d) = \bigl\{ u \in L^1(\mathbb R^d) \colon \|u\|_{n,1} < \infty \bigr\}, \\
    \|u\|_{n,1} = \max\limits_{|J| \leq n} \left\| \frac{\partial^{|J|} u}{\partial x^J}\right\|_{L^1(\mathbb R^d)}, \quad n \in \mathbb N \cup \{0\}.
  \end{gathered}
\end{equation}
More precisely, the approximation $u(\cdot,E)$ in \eqref{er.utov} is defined by
\begin{equation}\label{er.ucl}
    \begin{gathered}
    u(x,E) = \int\limits_{\mathcal B_{r(E)}} e^{-ipx} f(k_E(p),l_E(p)) \, dp, \quad x \in D, \\
    r(E) = 2\tau E^{\frac{\alpha}{n-d}} \quad \text{for some fixed $\tau \in (0,1]$},
    \end{gathered}
\end{equation}
where 
\begin{equation}\label{er.defBr}
    \mathcal B_r = \bigl\{ p \in \mathbb R^d \colon |p| \leq r \bigr\},
\end{equation}
$\alpha$ is defined in \eqref{er.utov}, and $k_E(p)$, $l_E(p)$ are defined as in \eqref{in.kElEdef} with some piecewise continuous vector-function $\gamma$ on $\mathbb R^d$; see, e.g., \cite{Nov2015en}. In addition, estimate \eqref{er.utov} can be precised as 
\begin{equation}
  |u(x,E) - v(x)| \leq A(D,N,M,d,n,\tau) E^{-\alpha}, \quad x \in D, \; E^{\frac 1 2} \geq \rho(D,N),
\end{equation}
where $\|v\|_{L^\infty(D)} \leq N$, $\|v\|_{n,1} \leq M$, $\rho$ is the same as in \eqref{in.estconsts} and the expression for $A$ can be found in formula (3.10) of \cite{Nov2015en}.

Analogs of $u(\cdot,E)$ for the phaseless case are given below in this section. In particular, related formulas depend on the zeros of determinant $\zeta_{\widehat w_1,\widehat w_2}$ of \eqref{in.zetanot0}.

We consider
\begin{gather}
  \begin{gathered}\label{er.defU}
    U_{\widehat w_1,\widehat w_2} = \Re U_{\widehat w_1,\widehat w_2} + i \Im U_{\widehat w_1,\widehat w_2}, \\
    \begin{pmatrix} 
        \Re U_{\widehat w_1,\widehat w_2}(p,E) \\
        \Im U_{\widehat w_1,\widehat w_2}(p,E)
    \end{pmatrix} = \tfrac 1 2 M_{\widehat w_1,\widehat w_2}^{-1}(p) b_{\widehat w_1,\widehat w_2}(p,E),
    \end{gathered}\\
    M_{\widehat w_1,\widehat w_2}(p) = \begin{pmatrix} 
                \Re \widehat w_1(p) & \Im \widehat w_1(p) \\
                \Re \widehat w_2(p) & \Im \widehat w_2(p)
              \end{pmatrix},\label{er.defM}\\
    M^{-1}_{\widehat w_1, \widehat w_2}(p) = \frac{1}{\zeta_{\widehat w_1,\widehat w_2}(p)} \begin{pmatrix} 
                \Im \widehat w_2(p) & -\Im \widehat w_1(p) \\ 
                -\Re \widehat w_2(p) & \Re \widehat w_1(p) 
                \end{pmatrix},\label{er.defMi}\\
    b_{\widehat w_1,\widehat w_2}(p,E) = \begin{pmatrix}
                   |f_1(p,E)|^2 - |f(p,E)|^2 - |\widehat w_1(p)|^2 \\
                   |f_2(p,E)|^2 - |f(p,E)|^2 - |\widehat w_2(p)|^2 
              \end{pmatrix}, \label{er.defb} \\
   f(p,E) = f(k_E(p),l_E(p)), \; f_j(p,E) = f_j(k_E(p),l_E(p)), \; j = 1,2, \label{er.fpviafkfl}
\end{gather}
where $\widehat w_1$, $\widehat w_2$, $f$, $f_1$, $f_2$ are the same as in \eqref{in.phlBorn}, \eqref{in.phlBornAbs}, $\zeta_{\widehat w_1,\widehat w_2}$ is defined by \eqref{in.zetanot0}, $k_E(p)$, $l_E(p)$ are the same as in \eqref{in.kElEdef}, \eqref{er.ucl}, and $p \in \mathcal B_{2\sqrt E}$, $d \geq 2$.

For Problem \ref{in.probpisc} for $d \geq 2$, $m = 2$, and for the case when $\zeta_{\widehat w_1,\widehat w_2}$ has no zeros (the case of \eqref{det.w2=iw1} in Section \ref{sec.zeros}) we have the following result:

\begin{theorem}\label{thm.w2=iw1} Let $v$ satisfy \eqref{in.vprop} and $v \in W^{n,1}(\mathbb R^d)$ for some $n > d$. Let $w_1$, $w_2$ be the same as in \eqref{det.w2=iw1}. Let
\begin{equation}\label{er.defr1}
    \begin{gathered}
    u(x,E) = \int\limits_{\mathcal B_{r_1(E)}} e^{-ipx} U_{\widehat w_1,\widehat w_2}(p,E) \, dp, \quad x \in D, \\
    r_1(E) = 2\tau E^\frac{\alpha_1}{n-d}, \quad \alpha_1 = \frac{n-d}{2(n+\beta)}, \quad \text{for some fixed $\tau \in (0,1]$},
    \end{gathered}
\end{equation}
where $U_{\widehat w_1,\widehat w_2}$ is defined by \eqref{er.defU}, $\mathcal B_r$ is defined by \eqref{er.defBr}, $\beta$ is the number of \eqref{det.wprop}, \eqref{det.w2=iw1}. Then 
\begin{equation}\label{er.t1uer}
  \begin{gathered}
    u(\cdot,E) = v + O(E^{-\alpha_1}) \quad \text{in $L^\infty(D)$}, \; E \to + \infty,\\
    |u(x,E)-v(x)| \leq A_1 E^{-\alpha_1},  \quad x \in D, \; E^{\frac 1 2} \geq \rho_1,
  \end{gathered}
\end{equation}
where $\rho_1$ and $A_1$ are defined in formulas \eqref{t1p.estvviaU} and \eqref{t1p.t1uer} of Section \ref{sec.t1p}.
\end{theorem}
Theorem \ref{thm.w2=iw1} is proved in Section \ref{sec.t1p}.

Next, we set
\begin{equation}\label{er.Zew1w2def}
    Z^\varepsilon_{\widehat w_1,\widehat w_2} = \bigl\{ p \in \mathbb R^d \colon py \in (-\varepsilon,\varepsilon) + \pi \mathbb Z \bigr\}, \quad y \in \mathbb R^d \setminus 0, \;\; 0 < \varepsilon < 1,
\end{equation}    
where $\widehat w_1$, $\widehat w_2$ and $y$ are the same as in \eqref{det.wjTj}--\eqref{det.defZw1w2}. One can see that $Z^\varepsilon_{\widehat w_1,\widehat w_2}$ is the open $\tfrac{\varepsilon}{|y|}$-neighborhood of $Z_{\widehat w_1,\widehat w_2}$ defined in \eqref{det.defZw1w2}.

Note that
\begin{equation}\label{er.defzt2}
  \begin{gathered}
 \text{for any $p \in Z^\varepsilon_{\widehat w_1,\widehat w_2}$ there exists}\\
  \text{the unique $z(p) \in \mathbb Z$ such that $|py - \pi z(p)| < \varepsilon$}.
  \end{gathered}
\end{equation}
In addition to $U_{\widehat w_1,\widehat w_2}$ of \eqref{er.defU}, we define
\begin{equation}\label{er.defUe}
  \begin{gathered}
    U^\varepsilon_{\widehat w_1,\widehat w_2}(p,E) = \tfrac 1 2 \bigl( U_{\widehat w_1,\widehat w_2}(p^\varepsilon_-,E)+U_{\widehat w_1,\widehat w_2}(p^\varepsilon_+,E) \bigr), \\
    p^\varepsilon_\pm = p_\bot + \pi z(p) \tfrac{y}{|y|^2} \pm \varepsilon \tfrac{y}{|y|^2}, \quad p_\bot = p - (py)\tfrac{y}{|y|^2}, \quad p \in \mathcal B_{2\sqrt E} \cap Z^\varepsilon_{\widehat w_1,\widehat w_2},
  \end{gathered}
\end{equation}
where $z(p)$ is the integer number of \eqref{er.defzt2}. The geometry of vectors $p$, $p_\bot$, $y$, $p^\varepsilon_\pm$ is illustrated in Fig. \ref{fig.defppm} for the case when the direction of $y$ coincides with the basis vector $e_1 = (1,0,\dots,0)$.

\begin{figure}
\begin{center}
\includegraphics[width=.7\linewidth]{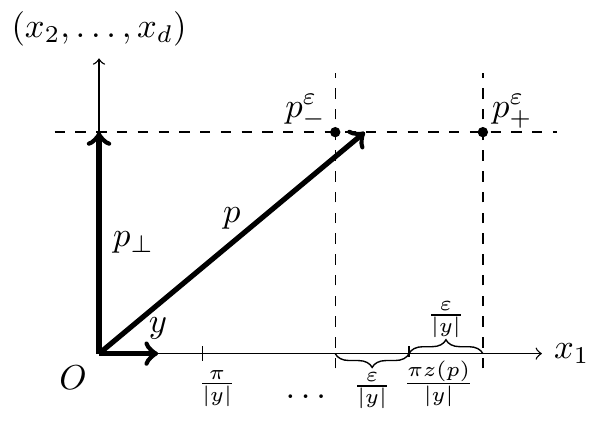}
\end{center}
\caption{Vectors $p$, $p_\bot$, $y$ and $p_\pm^\varepsilon$ of formula \eqref{er.defUe}}
\label{fig.defppm}
\end{figure}

For Problem \ref{in.probpisc} for $d \geq 2$, $m = 2$, and for the case when $\zeta_{\widehat w_1,\widehat w_2}$ has zeros on hyperplanes (the case of \eqref{det.defZw1w2} in Section \ref{sec.zeros}) we have the following result:

\begin{theorem}\label{thm.w2=w1sh} Let $v$ satisfy \eqref{in.vprop} and $v \in W^{n,1}(\mathbb R^d)$ for some $n > d$. Let $w_1$, $w_2$ be the same as in \eqref{det.wjTj}--\eqref{det.defZw1w2}. Let
\begin{equation}\label{er.defr2}
  \begin{gathered}
  u(x,E) = u_1(x,E) + u_2(x,E), \quad x \in D, \\
  u_1(x,E) = \int\limits_{ \mathcal B_{r_2(E)} \setminus Z^{\varepsilon_2(E)}_{\widehat w_1,\widehat w_2}} e^{-ipx} U_{\widehat w_1,\widehat w_2}(p,E) \, dp,  \\
  u_2(x,E) = \int\limits_{ \mathcal B_{r_2(E)} \cap Z^{\varepsilon_2(E)}_{\widehat w_1,\widehat w_2}} e^{-ipx}  U^\varepsilon_{\widehat w_1,\widehat w_2}(p,E) \, dp, \\
  r_2(E) = 2 \tau E^{\frac{\alpha_2}{n-d}}, \quad \varepsilon_2(E)= E^{-\tfrac{\alpha_2}{2}}, \\
   \alpha_2 = \tfrac{n-d}{2\bigl(n+\beta+\tfrac{n-d}{2}\bigr)}, \quad \text{for some fixed $\tau \in (0,1]$,}
  \end{gathered}
\end{equation}
where $U_{\widehat w_1,\widehat w_2}$ and $U^\varepsilon_{\widehat w_1,\widehat w_2}$ are defined by \eqref{er.defU}, \eqref{er.defUe}, $\mathcal B_r$ and $Z^\varepsilon_{\widehat w_1,\widehat w_2}$ are defined by \eqref{er.defBr}, \eqref{er.Zew1w2def}, and $\beta$ is the number of \eqref{det.wprop}. Then 
\begin{equation}\label{er.t2uer}
  \begin{gathered}
    u(\cdot,E) = v + O(E^{-\alpha_2}) \quad \text{in $L^\infty(D)$}, \; E \to + \infty, \\
    |u(x,E)-v(x)| \leq A_2 E^{-\alpha_2},  \quad x \in D, \; E^{\frac 1 2} \geq \rho_2,
  \end{gathered}
\end{equation}
where $\rho_2$ and $A_2$ are defined in formulas \eqref{t2p.estvviaU} and \eqref{t2p.t2uer} of Section \ref{sec.t2p}.
\end{theorem}
Theorem \ref{thm.w2=w1sh} is proved in Section \ref{sec.t2p}.

Next, we set 
\begin{equation}\label{er.Zew1..wd}
    \begin{gathered}
        Z^\varepsilon_{\widehat w_1,\ldots,\widehat w_{d+1}} = \mathcal B'_{\varepsilon/s} + \tfrac{\pi}{s}\mathbb Z^d, \quad 0 < \varepsilon < 1,\\
        \mathcal B'_r = \mathcal B_r \setminus \partial \mathcal B_r, \quad r>0,
    \end{gathered}
\end{equation}
where $\widehat w_1$, \dots, $\widehat w_{d+1}$ are the same as in \eqref{det.defZw1..wd}, \eqref{det.defw1wd}, and $\mathcal B_r$ is defined by \eqref{er.defBr}. One can see that $Z^\varepsilon_{\widehat w_1,\ldots,\widehat w_{d+1}}$ is the open $\frac{\varepsilon}{s}$-neighborhood of $Z_{\widehat w_1,\ldots,\widehat w_{d+1}}$ defined in \eqref{det.defZw1..wd}.

Note that 
\begin{equation}\label{er.defzt3}
    \begin{gathered}
        \text{for any $p \in Z^\varepsilon_{\widehat w_1,\ldots,\widehat w_{d+1}}$ there exists}\\
        \text{the unique $z(p) \in \mathbb Z^d$ such that $| sp - \pi z(p)| < \varepsilon$}.
    \end{gathered}
\end{equation}
In addition, we consider $i'$ such that
\begin{equation}\label{er.defi'}
    \begin{gathered}    
        i' = i'(p,s), \quad p = (p_1,\ldots,p_d) \in \mathbb R^d \setminus \tfrac{\pi}{s}\mathbb Z^d, \; s > 0,\\
    \text{$i'$ take values in $\{2,\ldots,d+1\}$},\\
    |\sin( s p_{i'-1} )| \geq |\sin(s p_{i-1})| \quad \text{for all $i \in \{2,\ldots,d+1\}$}.
    \end{gathered}
\end{equation}
Let 
\begin{gather}
    U_{\widehat w_1,\ldots,\widehat w_{d+1}}(p,E) = U_{\widehat w_1,\widehat w_{i'}}(p,E), \quad p \in \mathbb R^d \setminus \tfrac \pi s \mathbb Z^d, \label{er.defUw1..wd}\\
    \begin{gathered}\label{er.defUew1..wd}
    U^\varepsilon_{\widehat w_1,\ldots,\widehat w_{d+1}}(p,E) = \\
    \frac{1}{|\mathbb S^{d-1}|} \int\limits_{\mathbb S^{d-1}} U_{\widehat w_1,\ldots,\widehat w_{d+1}}(\tfrac \varepsilon s \vartheta + \tfrac{\pi}{s} z(p),E) \, d\vartheta, \quad
    p \in Z^\varepsilon_{\widehat w_1,\ldots,\widehat w_{d+1}},
    \end{gathered}\\
    \mathbb S^{d-1} = \bigl\{ p \in \mathbb R^d \colon |p| = 1 \bigr\},
\end{gather}
where $|\mathbb S^{d-1}|$ denotes the standard Euclidean volume of $\mathbb S^{d-1}$, $U_{\widehat w_1,\widehat w_{i'}}$ is defined in a similar way with $U_{\widehat w_1,\widehat w_2}$ of \eqref{er.defU}, $\widehat w_1$, \dots, $\widehat w_{d+1}$ are the same as in \eqref{det.defZw1..wd}, \eqref{det.defw1wd}, and $i'$ is the same as in \eqref{er.defi'}.

For Problem \ref{in.probpisc} for $d \geq 2$, $m = d+1$, we also have the following result:

\begin{theorem}\label{thm.w1..wd} Let $v$ satisfy \eqref{in.vprop} and $v \in W^{n,1}(\mathbb R^d)$ for some $n > d$. Let $w_1$, \dots, $w_{d+1}$ be the same as in \eqref{det.defw1wd}. Let
\begin{equation}\label{er.defr3}
  \begin{gathered}
  u(x,E) = u_1(x,E) + u_2(x,E), \quad x \in D, \\
  u_1(x,E) = \int\limits_{ \mathcal B_{r_3(E)} \setminus Z^{\varepsilon_3(E)}_{\widehat w_1,\dots,\widehat w_{d+1}}} e^{-ipx} U_{\widehat w_1,\dots,\widehat w_{d+1}}(p,E) \, dp,  \\
  u_2(x,E) = \int\limits_{ \mathcal B_{r_3(E)} \cap Z^{\varepsilon_3(E)}_{\widehat w_1,\dots,\widehat w_{d+1}}} e^{-ipx}  U^\varepsilon_{\widehat w_1,\dots,\widehat w_{d+1}}(p,E) \, dp, \\
  r_3(E) = 2 \tau E^{\frac{\alpha_3}{n-d}}, \quad \varepsilon_2(E) = E^{-\frac{\alpha_3}{d+1}},\\ 
  \alpha_3 = \tfrac{n-d}{2\bigl(n+\beta+\tfrac{n-d}{d+1}\bigr)}, \quad \text{for some fixed $\tau \in (0,1]$,}
  \end{gathered}
\end{equation}
where $U_{\widehat w_1,\dots,\widehat w_{d+1}}$ and $U^\varepsilon_{\widehat w_1,\dots,\widehat w_{d+1}}$ are defined by \eqref{er.defUw1..wd}, \eqref{er.defUew1..wd},  $\mathcal B_r$ and $Z^\varepsilon_{\widehat w_1,\dots,\widehat w_{d+1}}$ are defined by \eqref{er.defBr}, \eqref{er.Zew1..wd}, and $\beta$ is the number of \eqref{det.wprop}. Then 
\begin{equation}\label{er.t3uer}
  \begin{gathered}
    u(\cdot,E) = v + O(E^{-\alpha_3}) \quad \text{in $L^\infty(D)$}, \; E \to + \infty, \\
    |u(x,E)-v(x)| \leq A_3 E^{-\alpha_3},  \quad x \in D, \; E^{\frac 1 2} \geq \rho_3,
  \end{gathered}
\end{equation}
where $\rho_3$ and $A_3$ are defined in formulas \eqref{t3p.estvviaU} and \eqref{t3p.t3uer} of Section \ref{sec.t3p}.
\end{theorem}
Theorem \ref{thm.w1..wd} is proved in Section \ref{sec.t3p}.

\section{Proof of Theorem \ref{thm.w2=iw1}}\label{sec.t1p}

\begin{proposition}\label{prop.Uest} Let $v$ satisfy \eqref{in.vprop} and $w_1$, $w_2$ be the same as in \eqref{det.w2=iw1}, $d \geq 2$. Then:
\begin{equation}\label{t1p.estvviaU}
    \begin{gathered}
        \bigl| \widehat v(p) - U_{\widehat w_1,\widehat w_2}(p,E) \bigr| \leq c_2 |\widehat w(p)|^{-1} E^{-\frac 1 2} \quad \text{for $p \in \mathcal B_{2\sqrt E}$, $E^{\frac 1 2} \geq \rho_1$}, \\
        c_2 = 2c(D_0)N_0^3+2c(D_1)N_1^3,\\
        \rho_1 = \max_{j=0,1} \rho(D_j,N_j),
    \end{gathered}
\end{equation}
where $w$ is the function of \eqref{det.wprop}, \eqref{det.wjTj}, and $c$, $\rho$, $N_j$, $D_j$, $j = 0$, $1$, $2$, are the same as in estimates \eqref{in.estconsts}.
\end{proposition}
Proposition \ref{prop.Uest} follows from formulas \eqref{in.phlBorn}, \eqref{in.kElEdef}, estimates \eqref{in.estconsts}, definitions \eqref{er.defU}--\eqref{er.fpviafkfl}, and the properties that
\begin{equation}\label{t1p.*2=*1}
    \Omega_2 = \Omega_1, \; D_2 = D_1, \; N_2 = N_1.
\end{equation}
In turn, properties \eqref{t1p.*2=*1} follow from \eqref{in.defvj} for $j = 1$, $2$, \eqref{in.wjass} for $j = 1$, and from the equality $w_2 = i w_1$ assumed in \eqref{det.w2=iw1}.

Next, we represent $v$ as follows:
\begin{equation}\label{t1p.vdec}
    \begin{gathered}
        v(x) = v^+(x,r)+v^-(x,r), \quad x \in D, \; r > 0,\\
        v^+(x,r) = \int_{\mathcal B_r} e^{-ipx} \widehat v(p) \, dp, \\
        v^-(x,r) = \int_{\mathbb R^d \setminus \mathcal B_r} e^{-ipx} \widehat v(p) \, dp.
    \end{gathered}   
\end{equation}
Since $v \in W^{n,1}(\mathbb R^d)$, $n > d$, we have 
\begin{equation}\label{t1p.v-est}
  \begin{gathered}
        |v^-(x,r)| \leq c_3 \|v\|_{n,1} r^{d-n}, \quad x \in D, \; r > 0, \\
        c_3 = |\mathbb S^{d-1}|\tfrac{(2\pi)^{-d} d^n}{n-d},
  \end{gathered}
\end{equation}
where $\|\cdot\|_{n,1}$ is defined in \eqref{er.defWn1}, and $|\mathbb S^{d-1}|$ is the standard Euclidean volume of $\mathbb S^{d-1}$. Indeed,
\begin{equation}\label{t1p.esthatv}
  \begin{gathered}
  |p_1^{k_1} \cdots p_d^{k_d} \widehat v(p)| \leq (2\pi)^{-d}\|v\|_{n,1}, \quad p = (p_1,\ldots,p_d) \in \mathbb R^d,\\
  \text{for any $k_1$, \dots $k_d \in \mathbb N \cup \{0\}$, $k_1+\cdots+k_d \leq n$},
  \end{gathered}
\end{equation}
assuming also that $p_j^0 = 1$. Taking an appropriate sum in \eqref{t1p.esthatv} over all such $k_1$, \dots, $k_d$ with $k_1+\cdots+k_d = m \leq n$, we get
\begin{equation}\label{t1p.ineqhatv}
  |p|^m |\widehat v(p)| \leq (|p_1|+\cdots+|p_d|)^m |\widehat v(p)| \leq (2\pi)^{-d} d^m \|v\|_{n,1}, \quad p \in \mathbb R^d.
\end{equation}
The definition of $v^-$ of \eqref{t1p.vdec} and inequalities \eqref{t1p.ineqhatv} for $m = n$ imply \eqref{t1p.v-est}.

In addition, using Proposition \ref{prop.Uest} and the estimate on $\widehat w$ of \eqref{det.wprop}, we obtain:
\begin{equation}\label{t1p.v+est}
\begin{gathered}
        \left| v^+(x,r)-\int_{\mathcal B_r} e^{-ipx} U_{\widehat w_1,\widehat w_2}(p,E) \, dp \right| \leq c_2 E^{-\frac 1 2} \int_{\mathcal B_r} |\widehat w(p)|^{-1} \, dp \\
        \leq c_1^{-1} c_2 E^{-\frac 1 2} \int_{\mathcal B_r} (1+|p|)^\beta dp \leq c_4 E^{-\frac 1 2} r^{d+\beta}, \quad c_4 = |\mathbb S^{d-1}| \tfrac{2^{d+\beta}}{d+\beta} c_1^{-1} c_2, \\
        x \in D, \; 1 \leq r \leq 2 E^{\frac 1 2}, \; E^{\frac 1 2} \geq \rho_1.
\end{gathered}
\end{equation}
As a corollary of \eqref{t1p.vdec}, \eqref{t1p.v-est}, \eqref{t1p.v+est}, we have
\begin{equation}\label{t1p.eru1}
         \left|v(x) - \int_{\mathcal B_r} e^{-ipx} U_{\widehat w_1,\widehat w_2}(p,E) \, dp\right| \leq c_3\|v\|_{n,1} r^{d-n} + c_4 E^{-\frac 1 2} r^{d+\beta},
\end{equation}
where $x \in D$, $1 \leq r \leq 2  E^{\frac 1 2}$, $E^{\frac 1 2} \geq \rho_1$. In addition, if $r = r_1(E)$, where $r_1(E)$ is defined in \eqref{er.defr1}, then
\begin{equation}\label{t1p.rviaE}
    \begin{aligned}
        r^{d-n} & = (2\tau)^{d-n} E^{-\alpha_1}, \\
        E^{-\frac 1 2} r^{d+\beta} & = (2\tau)^{d+\beta} E^{-\alpha_1}.
    \end{aligned}
\end{equation}
Using formulas \eqref{t1p.eru1} and \eqref{t1p.rviaE} and taking into account definitions \eqref{er.defr1}, we obtain
\begin{equation}\label{t1p.t1uer}
  \begin{gathered}
    |u(x,E) - v(x)| \leq A_1 E^{-\alpha_1},  \quad x \in D, \; E^{\frac 1 2} \geq \rho_1 \\
    A_1 = A_1(D_0,D_1,N_0,N_1,M,d,n,\beta,\tau) \\
     = (2\tau)^{d-n} c_3 \|v\|_{n,1} + (2\tau)^{d+\beta} c_4,
  \end{gathered}
\end{equation}
where $D_j$, $N_j$, $j=0$, $1$, are the same as in estimates \eqref{in.estconsts} and $\|v\|_{n,1} \leq M$. 

Theorem \ref{thm.w2=iw1} is proved.

\section{Proof of Theorem \ref{thm.w2=w1sh}}\label{sec.t2p}

\begin{proposition}\label{prop.Uest2} Let $v$ satisfy \eqref{in.vprop} and $w_1$, $w_2$ be the same as in \eqref{det.wjTj}--\eqref{det.defZw1w2}, $d \geq 2$. Then:
\begin{equation}
    \begin{gathered}\label{t2p.estvviaU}
        \bigl| \widehat v(p) - U_{\widehat w_1,\widehat w_2}(p,E) \bigr| \leq c_5 \varepsilon^{-1} (1+|p|)^\beta E^{-\frac 1 2},\\
        p \in \mathcal B_{2\sqrt E}\setminus Z^\varepsilon_{\widehat w_1,\widehat w_2}, \; E^{\frac 1 2} \geq \rho_2, \; 0 < \varepsilon < 1,\\
      c_5 = \tfrac{\pi}{2} (2c(D_0)N_0^3 + c(D_1)N_1^3 + c(D_2)N_2^3) c_1,\\
      \rho_2 = \max_{j=0,1,2} \rho(D_j,N_j), 
   \end{gathered}
\end{equation}
in addition, if $v \in W^{n,1}(\mathbb R^d)$, $n \geq 0$, then:
\begin{gather}
   \begin{gathered}\label{t2p.estvviaUe}
     |\widehat v(p) - U^\varepsilon_{\widehat w_1,\widehat w_2}(p,E)| \\ 
    \leq 2^\beta c_5 \varepsilon^{-1} (1+|p|)^\beta E^{-\frac 1 2} + c_6\varepsilon (1+ \tfrac{\pi}{|y|}|z(p)| + |p_\bot| )^{-n},\\
        p \in \mathcal B_{2\sqrt E} \cap Z^\varepsilon_{\widehat w_1,\widehat w_2}, \; E \geq \rho_2, \; 0 < \varepsilon < \min\{1,\tfrac 1 2 |y|\}; 
   \end{gathered}\\
       c_6 = 2^n \tfrac{(d+1)^{n+1}}{(2\pi)^d |y|} \max_{j=1,\dots,d} \|x_j v\|_{n,1}, \notag
\end{gather}
where $c$, $\rho$, $N_j$, $D_j$, $j = 0$, $1$, $2$, are the same as in estimates \eqref{in.estconsts}, $c_1$, $\beta$ are the same as in Lemma \ref{det.lemw}; $z(p)$, $p_\bot$ are defined in \eqref{er.defzt2}, \eqref{er.defUe}, $x_j v = x_j v(x)$, and $\|\cdot\|_{n,1}$ is defined in \eqref{er.defWn1}.
\end{proposition}
\begin{proof}[Proof of Proposition \ref{prop.Uest2}] It follows from formulas \eqref{det.wjTj}, \eqref{det.defZw1w2} and \eqref{er.defM}, \eqref{er.defMi} that
\begin{equation}\label{t2p.expMMi}
  \begin{gathered}
  M_{\widehat w_1,\widehat w_2}(p) = \widehat w(p)
  \begin{pmatrix}
    \cos(T_1 p) & \sin (T_1 p) \\
    \cos(T_2 p) & \sin (T_2 p)
  \end{pmatrix}, \quad p \in \mathbb R^d,\\
  M^{-1}_{\widehat w_1,\widehat w_2}(p) = \frac{1}{\sin(py)\widehat w(p)} 
  \begin{pmatrix}
    \sin(T_2 p) & - \sin(T_1 p) \\
    -\cos(T_2 p) & \cos(T_1 p)
  \end{pmatrix}, \quad p \in \mathbb R^d \setminus Z^\varepsilon_{\widehat w_1,\widehat w_2}.
  \end{gathered}
\end{equation}
Also note that
\begin{equation}\label{t2p.estsin}
    |\sin (py)| \geq \tfrac{2\varepsilon}{\pi}, \quad p \in \mathbb R^d \setminus Z^\varepsilon_{\widehat w_1,\widehat w_2}, \; 0 < \varepsilon < 1.
\end{equation}
The estimate \eqref{t2p.estvviaU} follows from \eqref{in.phlBorn}, \eqref{in.estconsts}, \eqref{det.wprop}, \eqref{er.defU}, \eqref{er.defb}, \eqref{er.fpviafkfl} and from \eqref{t2p.expMMi}, \eqref{t2p.estsin}. 

It remains to prove \eqref{t2p.estvviaUe}. Using definition \eqref{er.defUe}, one can write
\begin{equation}\label{t2p.Uedec}
  \begin{gathered}
  U_{\widehat w_1,\widehat w_2}^\varepsilon(p,E) - \widehat v(p) = \varphi_1^\varepsilon(p,E) + \varphi_2^\varepsilon(p), \quad p \in \mathcal B_{2\sqrt E} \cap Z^\varepsilon_{\widehat w_1,\widehat w_2},\\
  \varphi_1^\varepsilon(p,E) = \tfrac 1 2 \bigl(U_{\widehat w_1,\widehat w_2}^\varepsilon(p_-^\varepsilon,E)-\widehat v(p_-^\varepsilon) \bigr) \\
     + \tfrac 1 2 \bigl( U_{\widehat w_1,\widehat w_2}^\varepsilon(p_+^\varepsilon,E)-\widehat v(p_+^\varepsilon)\bigr), \quad p \in \mathcal B_{2\sqrt E} \cap Z^\varepsilon_{\widehat w_1,\widehat w_2},\\
  \varphi_2^\varepsilon(p) = \tfrac 1 2 \bigl( \widehat v(p_-^\varepsilon) + \widehat v(p_+^\varepsilon) \bigr) - \widehat v(p), \quad p \in Z^\varepsilon_{\widehat w_1,\widehat w_2}.
  \end{gathered}
\end{equation}
Using estimate \eqref{t2p.estvviaU}, formula \eqref{t2p.Uedec} and the definitions of $p^\varepsilon_\pm$ in \eqref{er.defUe}, we get
\begin{equation}\label{t2p.phi1est}
    \begin{gathered}
  |\varphi^\varepsilon_1(p,E)| \leq \tfrac 1 2 c_5 \varepsilon^{-1} E^{-\frac 1 2} \bigl( (1+|p_-^\varepsilon|)^\beta + (1+|p_+^\varepsilon|)^\beta \bigr) \\
    \leq c_5 \varepsilon^{-1}(1+|p|+2\tfrac{\varepsilon}{|y|})^\beta E^{-\frac 1 2} 
  \leq 2^\beta c_5 \varepsilon^{-1}(1+|p|)^\beta E^{-\frac 1 2}, \\
        \text{for $\varepsilon$ as in \eqref{t2p.estvviaUe}, $p \in \mathcal B_{2\sqrt E} \cap Z^\varepsilon_{\widehat w_1,\widehat w_2}$.}
    \end{gathered}
\end{equation}
Next, using the definition of $\varphi^\varepsilon_2$ in \eqref{t2p.Uedec} and the mean value theorem, we obtain
\begin{equation}\label{t2p.phi2est}
   |\varphi_2^\varepsilon(p)| \leq \tfrac{\varepsilon}{|y|} \max \{ |\tfrac{y}{|y|} \nabla \widehat v(\xi)| \colon \xi \in [p_-^\varepsilon,p_+^\varepsilon] \}, \quad p \in Z^\varepsilon_{\widehat w_1,\widehat w_2},
\end{equation}
where $[p_-^\varepsilon,p_+^\varepsilon]$ denotes the segment joining $p_-^\varepsilon$ to $p_+^\varepsilon$. Here, the mean value theorem was used for $\widehat v(\xi)$ on $[p_-^\varepsilon,p]$ and on $[p,p_+^\varepsilon]$.

Note also that
\begin{equation}\label{t2p.dvest}
|\nabla \widehat v(\xi)| \leq d \max_{j=1,\dots,d} \bigl| \tfrac{\partial \widehat v}{\partial \xi_j}(\xi) \bigr|, \quad \xi = (\xi_1,\dots,\xi_d) \in [p_-^\varepsilon,p_+^\varepsilon].
\end{equation}
In addition, the following estimates hold:
\begin{equation}\label{t2p.phi2gest}
  \begin{gathered}
   \biggl|\frac{\partial \widehat v}{\partial \xi_j}(\xi) \biggr| \leq \frac{(1+d)^n}{(2\pi)^d (1+|\xi|)^n} \| x_j v \|_{n,1}, \quad \xi \in [p_-^\varepsilon,p_+^\varepsilon], \; j = 1,\dots,d.
  \end{gathered}
\end{equation}
Indeed, taking the sum in \eqref{t1p.ineqhatv} over all $m = 0$, \dots, $n$ with the binomial coefficients, we get
\begin{equation}\label{t2p.ineqhatv}
  \begin{gathered}
  (1+|p|)^n |\widehat v(p)| \leq (1+|p_1|+\cdots+|p_d|)^n |\widehat v(p)| \\
     \leq (2\pi)^{-d} (1+d)^n \|v\|_{n,1}, \quad p \in \mathbb R^d.
  \end{gathered}
\end{equation}
Estimates \eqref{t2p.phi2gest} follow from \eqref{t2p.ineqhatv}, where we replace $v$ by $x_j v$ and use that $v$ belongs to $W^{n,1}(\mathbb R^d)$ and is compactly supported.

Estimates \eqref{t2p.phi2est}--\eqref{t2p.phi2gest} imply
\begin{equation}\label{t2p.phi2festmax}
  |\varphi^\varepsilon_2(p)| \leq 2^{-n} c_6 \varepsilon \max\bigl\{ \bigl(1+ |\xi| \bigr)^{-n} \colon \xi \in [p_-^\varepsilon,p_+^\varepsilon] \bigr\}, \quad p \in Z^\varepsilon_{\widehat w_1,\widehat w_2}.
\end{equation}
Using also that
\begin{equation}\label{t2p.pviataupb}
    \text{$\xi = \tau \tfrac{y}{|y|} + p_\bot$, where $|\tau - \tfrac{\pi}{|y|}z(p) | \leq \tfrac{\varepsilon}{|y|}$, if $\xi \in [p_-^\varepsilon,p_+^\varepsilon]$,}
\end{equation}
and that $\varepsilon < |y|$, we obtain
\begin{equation}\label{t2p.phi2fest}
  \begin{gathered}
  |\varphi^\varepsilon_2(p)| \leq 2^{-n} c_6 \varepsilon \bigl(1+\tfrac 1 2 ( \tfrac{\pi}{|y|} |z(p)| - \tfrac{\varepsilon}{|y|} + |p_\bot|) \bigr)^{-n} \\
    \leq c_6 \bigl(1+ \tfrac{\pi}{|y|} |z(p)| + |p_\bot| \bigr)^{-n}, \quad p \in Z^\varepsilon_{\widehat w_1,\widehat w_2}.
  \end{gathered}
\end{equation}
Estimate \eqref{t2p.estvviaUe} follows from \eqref{t2p.phi1est} and \eqref{t2p.phi2fest}.

Proposition \ref{prop.Uest2} is proved. 
\end{proof}

The final part of the proof of Theorem \ref{thm.w2=w1sh} is as follows. In a similar way with \eqref{t1p.vdec}, we represent $v$ as follows:
\begin{equation}\label{t2p.vdec}
 \begin{gathered}
    v(x) = v^+_1(x,r) + v^+_2(x,r) + v^-(x,r), \quad x \in D, \; r > 0,\\
    v^+_1(x,r) = \int_{\mathcal B_r \setminus Z^\varepsilon_{\widehat w_1,\widehat w_2}} e^{-ipx} \widehat v(p) \, dp, \\
    v^+_2(x,r) = \int_{\mathcal B_r \cap Z^\varepsilon_{\widehat w_1,\widehat w_2}} e^{-ipx} \widehat v(p) \, dp, \\
    v^-(x,r) = \int_{\mathbb R^d \setminus \mathcal B_r} e^{-ipx} \widehat v(p) \, dp.
 \end{gathered}
\end{equation}
Since $v \in W^{n,1}(\mathbb R^d)$, estimate \eqref{t1p.v-est} holds.

Using estimates \eqref{t2p.estvviaU}, \eqref{t2p.estvviaUe}, we get:
\begin{gather}\label{t2p.v+est}
  \begin{gathered}
      \biggl| v^+_1(x,r)-\int_{\mathcal B_r \setminus Z^\varepsilon_{\widehat w_1,\widehat w_2}}   e^{-ipx} U_{\widehat w_1,\widehat w_2}(p,E) \, dp \\
        + v^+_2(x,r) - \int_{\mathcal B_r \cap Z^\varepsilon_{\widehat w_1,\widehat w_2} }   e^{-ipx} U^\varepsilon_{\widehat w_1,\widehat w_2}(p,E) \, dp \biggr| \leq I_1 + I_2, \\
  \end{gathered}\\
   I_1 = 2^\beta c_5 \varepsilon^{-1} E^{-\frac 1 2} \int_{\mathcal B_r} (1+|p|)^\beta dp,\\
   I_2 = c_6\varepsilon \int_{\mathcal B_r \cap Z^\varepsilon_{\widehat w_1,\widehat w_2}} \hspace{-2em} \bigl(1+ \tfrac{\pi}{|y|} |z(p)| + |p_\bot| \bigr)^{-n} \, dp,\\
      x \in D, \; 1 \leq r \leq 2 E^{\frac 1 2}, \; E^{\frac 1 2} \geq \rho_2, \notag
\end{gather}
where $\rho_2$ is the same as in Proposition \ref{prop.Uest2}. In addition:
\begin{gather}
    \begin{gathered}\label{t2p.I1est}
    I_1 \leq c_7 \varepsilon^{-1} E^{\frac 1 2} r^{d+\beta}, \quad c_7 = |\mathbb S^{d-1}| \tfrac{2^{d+2\beta}}{d+\beta} c_5;
    \end{gathered}\\
    I_2 = c_6 \varepsilon \sum_{z \in \mathbb Z} \int\limits_{ \left\{ \begin{smallmatrix} \tau^2 + p_\bot^2 \leq r^2 \\  |\tau - \tfrac{\pi}{|y|}z| \leq \tfrac{\varepsilon}{|y|} \end{smallmatrix} \right\} } \hspace{-2em}  \bigl( 1+ \tfrac{\pi}{|y|}|z|+ |p_\bot| \bigr)^{-n} d\tau dp_\bot, \notag
\end{gather}
where $\tau \in \mathbb R$, $p_\bot \in \mathbb R^d$, $p_\bot \cdot y = 0$, 
\begin{equation}\label{t2p.I2est}
    \begin{gathered}
   I_2 \leq 2 c_6 \frac{\varepsilon^2}{|y|} \sum_{z \in \mathbb Z} \; \int\limits_{\xi \in \mathbb R^{d-1}, |\xi|\leq r} \hspace{-1em} \bigl(1+ \tfrac{\pi}{|y|} |z| + |\xi| \bigr)^{-n} \, d\xi \leq c_6 c_8 \varepsilon^2, \\
    c_8 =  \frac{2}{|y|}\frac{|\mathbb S^{d-2}|}{n-d+1} \sum_{z \in \mathbb Z} (1 + \tfrac{\pi}{|y|} |z| )^{d-n-1}.
    \end{gathered}
\end{equation}

In addition, if $r = r_2(E)$, $\varepsilon = \varepsilon_2(E)$, where $r_2(E)$, $\varepsilon_2(E)$ are defined in \eqref{er.defr2}, then
\begin{equation}\label{t2p.rviaE}
  \begin{aligned}
      r^{d-n} & = (2\tau)^{d-n} E^{-\alpha_2}, \\
      \varepsilon^{-1}E^{-\frac 1 2} r^{d+\beta} & = (2\tau)^{d+\beta} E^{-\alpha_2}, \\
      \varepsilon^2 & = E^{-\alpha_2}.
  \end{aligned}
\end{equation}

Using representation \eqref{t2p.vdec}, estimates \eqref{t1p.v-est}, \eqref{t2p.v+est}, \eqref{t2p.I1est}, \eqref{t2p.I2est}, formulas \eqref{t2p.rviaE} and taking into account definitions \eqref{er.defr2}, we obtain
\begin{equation}\label{t2p.t2uer} 
  \begin{gathered}
  |u(x,E)-v(x)| \leq A_2  E^{-\alpha_2}, \quad x \in D, \; E^{\frac 1 2} \geq \rho_2, \\
  A_2 = A_2(D_0,D_1,D_2,N_0,N_1,N_2,M,M_1,\dots,M_d,d,n,\beta,\tau,|y|) \\
   = (2\tau)^{d+\beta} c_7 + c_6 c_8 + (2\tau)^{d-n} c_3 \|v\|_{n,1},
  \end{gathered}
\end{equation}
where $D_j$, $N_j$ are the same as in estimates \eqref{in.estconsts} and $\|v\|_{n,1} \leq M$, $\|x_j v\|_{n,1} \leq M_j$.

Theorem \ref{thm.w2=w1sh} is proved.

\section{Proof of Theorem \ref{thm.w1..wd}}\label{sec.t3p}

\begin{proposition}\label{prop.Uest3}
  Let $v$ satisfy \eqref{in.vprop} and $w_1$, \dots, $w_{d+1}$ be the same as in \eqref{det.wjTj}, \eqref{det.defw1wd}, $d \geq 2$. Then:
  \begin{equation}
      \begin{gathered}\label{t3p.estvviaU}
     |\widehat v(p) - U_{\widehat w_1,\dots,\widehat w_{d+1}}(p,E)| \leq c_9 \varepsilon^{-1} (1+|p|)^\beta E^{-\frac 1 2}, \\
      p \in \mathcal B_{2\sqrt E} \setminus Z^\varepsilon_{\widehat w_1,\dots,\widehat w_{d+1}}, \; E^{\frac 1 2} \geq \rho_3, \; 0 < \varepsilon < 1,\\
        c_9 = \tfrac{\pi \sqrt d}{2} \bigl( 2c(D_0)N_0^3 + c(D_1)N_1^3 + \max_{j=2,\dots,d+1}c(D_j)N_j^3 \bigr)c_1,\\
    \rho_3 = \max_{j=0,\dots,d+1} \rho(D_j,N_j),
  \end{gathered}
\end{equation}
in addition, if $v \in W^{n,1}(\mathbb R^d)$, $n \geq 0$, then:
\begin{gather}
      \begin{gathered}\label{t3p.estvviaUe}
     |\widehat v(p) - U^\varepsilon_{\widehat w_1,\dots,\widehat w_{d+1}}(p,E)| \\
     \leq 2^\beta c_9 \varepsilon^{-1} (1+|p|)^\beta E^{-\frac 1 2} + 2 c_6 \varepsilon \bigl( 1 + 2 \tfrac \pi s \|z(p)\|_2 \bigr)^{-n},\\
      p \in \mathcal B_{2\sqrt E} \cap Z^\varepsilon_{\widehat w_1,\dots,\widehat w_{d+1}}, \; E^{\frac 1 2} \geq \rho_3, \; 0 < \varepsilon < \min\{1, \tfrac 1 2 s\},
      \end{gathered}
  \end{gather}
where $c$, $\rho$, $D_j$, $N_j$, $j = 0$, \dots, $d+1$, are defined as in \eqref{in.estconsts}; $c_1$, $\beta$ are the same as in Lemma \ref{det.lemw}, $c_6$ is the same as in Proposition \ref{prop.Uest2}, $z(p)$ is defined in \eqref{er.defzt3} and $\| z(p) \|_2$ is the standard Euclidean norm of $z(p)$.
\end{proposition}
\begin{proof}[Proof of Proposition \ref{prop.Uest3}] In a similar way with formulas \eqref{t2p.expMMi}, one can write
\begin{equation}
  \begin{gathered}\label{t3p.expMMi}
    M_{\widehat w_1,\widehat w_{i'}}(p) = \widehat w_1(p)
      \begin{pmatrix} 
	 1 & 0 \\
	 \cos(s p_{i'-1}) & \sin(s p_{i'-1})
      \end{pmatrix}, \quad p \in \mathbb R^d \setminus \tfrac{\pi}{s} \mathbb Z^d, \\
    \hspace{-1em} M^{-1}_{\widehat w_1,\widehat w_{i'}}(p) = \frac{1}{\sin(s p_{i'-1}) \widehat w_1(p)} 
      \begin{pmatrix}
	\sin(s p_{i'-1}) & 0 \\
	-\cos(s p_{i'-1}) & 1
      \end{pmatrix}, \; p \in \mathbb R^d \setminus Z^\varepsilon_{\widehat w_1,\dots,\widehat w_{d+1}},
   \end{gathered}
\end{equation}
where $i' = i'(p,s)$ is defined in \eqref{er.defi'}. Also note that
\begin{equation}\label{t3p.estsin}
  |\sin(s p_{i'-1})| \geq \tfrac{2\varepsilon}{\pi \sqrt d} , \quad p \in \mathbb R^d \setminus Z^\varepsilon_{\widehat w_1,\dots,\widehat w_{d+1}}, \; 0 < \varepsilon < 1.
\end{equation}
Estimate \eqref{t3p.estvviaU} follows from \eqref{in.phlBorn}, \eqref{in.estconsts}, \eqref{det.wprop}, \eqref{er.defU}, \eqref{er.defb}, \eqref{er.fpviafkfl}, \eqref{er.defi'}, \eqref{er.defUw1..wd} and from \eqref{t3p.expMMi}, \eqref{t3p.estsin}.

It remains to prove \eqref{t3p.estvviaUe}. Using definition \eqref{er.defUew1..wd}, we represent
\begin{equation}\label{t3p.Uedec}
  \begin{gathered}
  U^\varepsilon_{\widehat w_1,\dots,\widehat w_{d+1}}(p,E) - \widehat v(p) = \varphi_1^\varepsilon(p,E) + \varphi_2^\varepsilon(p), \quad p \in \mathcal B_{2\sqrt E} \cap Z^\varepsilon_{\widehat w_1,\dots,\widehat w_{d+1}} \\
  \varphi_1^\varepsilon(p,E) = \frac{1}{|\mathbb S^{d-1}|} \int_{\mathbb S^{d-1}} \bigl( U_{\widehat w_1,\dots,\widehat w_{d+1}}(\eta,E) - \widehat v(\eta) \bigr)\bigr|_{\eta = \tfrac{\varepsilon}{s}\vartheta + \tfrac{\pi}{s}z(p)} d\vartheta,\\
  \varphi_2^\varepsilon(p) = \frac{1}{|\mathbb S^{d-1}|} \int_{\mathbb S^{d-1}} \bigl( \widehat v\bigl( \tfrac{\varepsilon}{s}\vartheta + \tfrac{\pi}{s} z(p) \bigr) - \widehat v(p) \bigr) d\vartheta,
  \end{gathered}
\end{equation}
where $z(p)$ is defined in \eqref{er.defzt3}.

Using formulas \eqref{t3p.estvviaU}, \eqref{t3p.Uedec}, we obtain
\begin{equation}\label{t3p.phi1est}
  \begin{gathered}
    |\varphi_1^\varepsilon(p,E)| \leq c_9 \varepsilon^{-1} E^{-\frac 1 2} \frac{1}{|\mathbb S^{d-1}|} \int_{\mathbb S^{d-1}} \bigl( 1 + \bigl| \tfrac \varepsilon s \vartheta + \tfrac \pi s z(p) \bigr| \bigr)^\beta d\vartheta \\
    \leq c_9 \varepsilon^{-1} E^{-\frac 1 2} \frac{1}{|\mathbb S^{d-1}|} \int_{\mathbb S^{d-1}} \bigl(1 + |p| + 2\tfrac \varepsilon s \bigr)^\beta d\vartheta \leq 2^\beta c_9 \varepsilon^{-1} (1+|p|)^\beta E^{-\frac 1 2}, \\
      \text{for $\varepsilon$ as in \eqref{t3p.estvviaUe}}.
  \end{gathered}
\end{equation}

Next, using the definition of $\varphi^\varepsilon_2$ in formula \eqref{t3p.Uedec} and the mean value theorem, we get the following estimate:
\begin{equation}\label{t3p.phi2estmax}
 |\varphi^\varepsilon_2(p)| \leq 2 \tfrac{\varepsilon}{s} \max\bigl\{ |\nabla \widehat v(\xi) | \colon \xi \in \mathbb R^d, \; |\xi - \tfrac{\pi}{s} z(p)| \leq \tfrac{\varepsilon}{s} \bigr\}, \; p \in Z^\varepsilon_{\widehat w_1,\dots,\widehat w_{d+1}}.
\end{equation}
Here, the mean value theorem was used for $\widehat v(\xi)$ on $[p,\tfrac \varepsilon s \vartheta + \tfrac \pi s z(p)]$, $\vartheta \in \mathbb S^{d-1}$. One can see that 
\begin{equation}\label{t3p.dvest}
  \begin{gathered}
  \text{estimates \eqref{t2p.dvest} and \eqref{t2p.phi2gest} hold for all $\xi \in \mathbb R^d$}\\
  \text{such that $|\xi - \tfrac \pi s z(p)| \leq \tfrac \varepsilon s$, where $p \in Z^\varepsilon_{\widehat w_1,\dots, \widehat w_{d+1}}$}.
  \end{gathered}
\end{equation}
It follows from \eqref{t3p.phi2estmax}, \eqref{t3p.dvest} and from the upper estimate on $\varepsilon$ of \eqref{t3p.estvviaUe}, that
\begin{equation}\label{t3p.phi2est}
  \begin{gathered}
  |\varphi_2^\varepsilon(p)| \leq 2^{1-n} c_6 \varepsilon  \max\bigl\{ (1+|\xi|)^{-n} \colon |\xi - \tfrac \pi s z(p)| \leq \tfrac \varepsilon s \bigr\} \\
    \leq 2^{1-n} c_6 \varepsilon \bigl(1+\tfrac \pi s \|z(p)\|_2 - \tfrac \varepsilon s\bigr)^{-n} \\
    \leq 2 c_6 \varepsilon \bigl(1 + 2\tfrac \pi s \|z(p)\|_2 \bigr)^{-n}, \; p \in Z^\varepsilon_{\widehat w_1,\dots,\widehat w_{d+1}}.
  \end{gathered}
\end{equation}
Estimate \eqref{t3p.estvviaUe} follows from estimates \eqref{t3p.phi1est} and \eqref{t3p.phi2est}.
 
Proposition \ref{prop.Uest3} is proved. 
\end{proof}

The final part of the proof of Theorem \ref{thm.w1..wd} is as follows. In a similar way with \eqref{t2p.vdec}, we represent $v$ as follows:
\begin{equation}\label{t3p.vdec}
 \begin{gathered}
    v(x) = v^+_1(x,r) + v^+_2(x,r) + v^-(x,r), \quad x \in D, \; r > 0,\\
    v^+_1(x,r) = \int_{\mathcal B_r \setminus Z^\varepsilon_{\widehat w_1,\dots,\widehat w_{d+1}}} e^{-ipx} \widehat v(p) \, dp, \\
    v^+_2(x,r) = \int_{\mathcal B_r \cap Z^\varepsilon_{\widehat w_1,\dots,\widehat w_{d+1}}} e^{-ipx} \widehat v(p) \, dp, \\
    v^-(x,r) = \int_{\mathbb R^d \setminus \mathcal B_r} e^{-ipx} \widehat v(p) \, dp.
 \end{gathered}
\end{equation}
Since $v$ belongs to $W^{n,1}(\mathbb R^d)$, estimate \eqref{t1p.v-est} is valid.

Using estimates \eqref{t3p.estvviaU}, \eqref{t3p.estvviaUe} we obtain
\begin{equation}\label{t3p.v+est}
  \begin{gathered}
      \biggl| v^+_1(x,r)-\int_{\mathcal B_r \setminus Z^\varepsilon_{\widehat w_1,\dots,\widehat w_{d+1}} } \hspace{-2em}  e^{-ipx} U_{\widehat w_1,\dots,\widehat w_{d+1}}(p,E) \, dp \\
        + v^+_2(x,r) - \int_{\mathcal B_r \cap Z^\varepsilon_{\widehat w_1,\dots,\widehat w_{d+1}}} \hspace{-2em} e^{-ipx} U^\varepsilon_{\widehat w_1,\dots,\widehat w_{d+1}}(p,E) \, dp  \biggr| \leq J_1 + J_2, \\
        J_1 = 2^\beta c_9 \varepsilon^{-1} E^{-\frac 1 2} \int_{\mathcal B_r} (1+|p|)^\beta dp, \\
        J_2 = 2 c_6 \varepsilon \int_{\mathcal B_r \cap Z_{\widehat w_1,\dots,\widehat w_{d+1}}} \hspace{-2em} \bigl(1+2 \tfrac \pi s \|z(p)\|_2\bigr)^{-n} dp,\\
      x \in D, \; 1 \leq r \leq 2 E^{\frac 1 2}, \; E^{\frac 1 2} \geq \rho_3,
  \end{gathered}
\end{equation}
where $\rho_3$ is the same as in Proposition \ref{prop.Uest3}. In addition,
\begin{equation}\label{t3p.Jest}
\begin{gathered}
    J_1 \leq c_{10} \varepsilon^{-1} E^{-\frac 1 2} r^{d+\beta}, \quad c_{10} = |\mathbb S^{d-1}| \tfrac{2^{d+\beta}}{d+\beta} c_9, \\
    J_2 \leq c_{11} \varepsilon^{d+1}, \quad c_{11} = \left(\tfrac 1 s\right)^d |\mathcal B_1| \sum_{z \in \mathbb Z^d} \bigl(1 + 2 \tfrac \pi s \|z\|_2 \bigr)^{-n},
\end{gathered}
\end{equation}
where $|\mathcal B_1|$ is the standard Euclidean volume of $\mathcal B_1$. Finally, if $r = r_3(E)$, $\varepsilon = \varepsilon_3(E)$, where $r_3(E)$, $\varepsilon_3(E)$ are defined in \eqref{er.defr3}, then
\begin{equation}\label{t3p.rviaE}
  \begin{aligned}
      r^{d-n} & = (2\tau)^{d-n} E^{-\alpha_3}, \\
      \varepsilon^{-1} E^{-\frac 1 2} r^{d+\beta} & = (2\tau)^{d+\beta} E^{-\alpha_3}, \\
      \varepsilon^{d+1} & = E^{-\alpha_3}.
  \end{aligned}
\end{equation}
Using representation \eqref{t3p.vdec}, estimates \eqref{t1p.v-est}, \eqref{t3p.v+est}, \eqref{t3p.Jest}, formulas \eqref{t3p.rviaE} and taking into account definitions \eqref{er.defr3}, we obtain
\begin{equation}\label{t3p.t3uer}
  \begin{gathered}
  |u(x,E) - v(x)| \leq A_3 E^{-\alpha_3}, \quad x \in D, \; E^{\frac 1 2} \geq \rho_3, \\
  A_3 = A_3(D_0,\dots,D_{d+1},N_0,\dots,N_{d+1},M,d,n,\beta,\tau,s) \\
 = (2\tau)^{d+\beta}c_{10} + c_{11} + (2\tau)^{d-n}c_3\|v\|_{n,1},
  \end{gathered}
\end{equation}
where $\|v\|_{n,1} \leq M$ and $D_0$, \dots, $D_{d+1}$, $N_0$, \dots, $N_{d+1}$ are the same as in Proposition \ref{prop.Uest3}. 

Theorem \ref{thm.w1..wd} is proved.

\section{Proof of Lemma \ref{det.lemw}}\label{sec.l1p}
Note that 
\begin{gather}
    \widehat w(p) = \int_{\mathbb R^d} |\widehat q(\xi)|^2 \widehat \omega_\nu(p-\xi) \, d\xi, \quad p \in \mathbb R^d, \label{l1p.wft}\\
    \omega_\nu(x) = |x|^\nu K_\nu\bigl( |x| \bigr), \quad x \in \mathbb R^d, \label{l1p.defomnu}
\end{gather}
where $\widehat q$, $\widehat \omega_\nu$ are the Fourier transforms of $q$, $\omega_\nu$. The Fourier transform $\widehat \omega_\nu$ can be computed explicitely:
\begin{equation}\label{l1p.omnuft}
  \widehat \omega_\nu(p) = \frac{c_{12}}{(1+|p|^2)^{\frac d 2 + \nu}}, \quad c_{12} = \frac{\Gamma(\tfrac d 2 + \nu) 2^{\nu-1}}{\pi^{\frac d 2}}.
\end{equation}
Indeed, formula \eqref{l1p.omnuft} follows from the Fourier inversion theorem and the following computations:
\begin{equation*}
  \begin{gathered}
  \int_{\mathbb R^d} \frac{e^{-ipx} \, dp}{(1+|p|^2)^{\frac d 2 + \nu}} = \int_{\mathbb R} \int_{\mathbb R^{d-1}} \frac{e^{-i|x|t} \, dt\,d\xi}{(1+t^2 + |\xi|^2)^{\frac d 2 + \nu}} \\
  = |\mathbb S^{d-2}| \int_{\mathbb R}\int_{\mathbb R} \frac{e^{-i|x|t} r^{d-2} \, dt \, dr}{(1+t^2+r^2)^{\frac d 2 + \nu}} \\
  \overset{r=\sqrt{1+t^2}\tau}{=\joinrel=\joinrel=\joinrel=\joinrel=\joinrel=} |\mathbb S^{d-2}| \int_{\mathbb R} \frac{e^{-i|x|t} \, dt}{(1+t^2)^{\frac 1 2 + \nu}} \int_0^{+\infty}\frac{\tau^{d-2} \, d\tau}{(1+\tau^2)^{\frac d 2 + \nu}} \\
     = c_{12}^{-1} |x|^\nu K_\nu(|x|), \; x \in \mathbb R^d.
  \end{gathered}
\end{equation*}
Here, it was used that
\begin{gather*}
      \int_0^{+\infty} \frac{\tau^{d-2} \, d\tau}{(1+\tau^2)^{\frac d 2 + \nu}} = \frac 1 2 B(\tfrac{d-1}{2},\nu+\tfrac 1 2) = \frac 1 2 \frac{\Gamma(\tfrac{d-1}{2}) \Gamma(\nu+\tfrac 1 2)}{\Gamma(\tfrac d 2 + \nu)}, \\
        |\mathbb S^{d-2}| = \frac{2\pi^{\frac{d-1}{2}}}{\Gamma(\tfrac{d-1}{2})},
\end{gather*}
where $B$ and $\Gamma$ denote the beta and gamma functions.

Using \eqref{l1p.wft}, \eqref{l1p.omnuft}, we obtain the estimates 
\begin{equation}\label{l1p.west}
  \begin{gathered}
 \widehat w(p) \geq \int_{|\xi| \leq 1} \frac{c_{12} |\widehat q(\xi)|^2}{(1+|p-\xi|)^{d+2\nu}} d\xi \geq \frac{c_1(q,\nu)}{(1+|p|)^{d+2\nu}}, \quad p \in \mathbb R^d, \\
  c_1(q,\nu) = \frac{c_{12}}{2^{d+2\nu}} \int_{|\xi| \leq 1} |\widehat q(\xi)|^2 \, d\xi. 
  \end{gathered}
\end{equation}
Properties \eqref{det.wprop} follow from \eqref{det.defw}, \eqref{det.qprop}, \eqref{l1p.wft} and \eqref{l1p.west}.

Lemma \ref{det.lemw} is proved.

\section*{Aknowledgements}
The authors are grateful to the referee for remarks that have helped to improve the presentation.

\bibliographystyle{IEEEtranS}
\bibliography{biblio_utf}

\end{document}